\documentclass[a4paper, 10 pt, conference]{ieeeconf}
\IEEEoverridecommandlockouts
\let\labelindent\relax

\usepackage{graphicx}
\usepackage{times}   % assumes new font selection scheme installed
\usepackage{amsmath} % assumes amsmath package installed
\usepackage{amssymb} % assumes amsmath package installed
\usepackage{color}
\usepackage{float}
\usepackage{verbatim}
\usepackage{mathtools}
\usepackage{enumitem}

\def\O{\dagger}
\def\hx{\hat{x}}
\def\hxj{\hat{x}^j}
\def\hxjp{\hat{x}^{j+1}}
\def\hy{\hat{y}}
\def\hyj{\hat{y}^j}
\def\hyjp{\hat{y}^{j+1}}
\def\S{\mathcal{S}}
\DeclareMathOperator{\diag}{diag}

\newtheorem{theorem}{Theorem}%[section]
\newtheorem{lemma}{Lemma}%[section]
\newtheorem{assumption}{Assumption}%[section]

\allowdisplaybreaks

\title{\LARGE \bf
ADMM for MPC with state and input constraints, and input nonlinearity
}

\author{Sebastian East$^{1}$,~~Mark Cannon$^{1}$% <-this % stops a space
\thanks{$^{1}$Department of Engineering Science, University
  of Oxford, OX1\,3PJ, UK}}

\begin{document}
\maketitle
\thispagestyle{empty}
\pagestyle{empty}

%%%%%%%%%%%%%%%%%%%%%%%%%%%%%%%%%%%%%%%%%%%%%%%%%%%%%%%%%%%%%%%%%%%%%%%%%%%%%%%%
\begin{abstract}
In this paper we propose an Alternating Direction Method of Multipliers (ADMM) algorithm for solving a Model Predictive Control (MPC) optimization problem, in which the system has state and input constraints and a nonlinear input map. The resulting optimization is nonconvex, and we provide a proof of convergence to a point satisfying necessary conditions for optimality. This general method is proposed as a solution for blended mode control of hybrid electric vehicles, to allow optimization in real time. To demonstrate the properties of the algorithm we conduct numerical experiments on randomly generated problems, and show that the algorithm is effective for achieving an approximate solution, but has limitations when an exact solution is required.
\end{abstract}

%%%%%%%%%%%%%%%%%%%%%%%%%%%%%%%%%%%%%%%%%%%%%%%%%%%%%%%%%%%%%%%%%%%%%%%%%%%%%%%%
\section{INTRODUCTION}

In Model Predictive Control (MPC), control inputs are optimized by solving an open-loop optimal control problem over a finite prediction horizon, whilst explicitly accounting for constraints on system states and inputs. This approach is widely used in process control applications, in which slow plant dynamics allow the optimization to be solved in real time, and the control of mechanical and electrical systems is becoming more feasible with improvements in embedded controller hardware \cite{Lee2011}. To further increase MPC applicability, recent research has considered exploiting properties of optimization problems to improve computational efficiency.

Active set \cite{Ferreau2008} and interior point methods \cite{Domahidi2012} are commonly employed in MPC, but for reduced computational complexity first order approaches such as fast gradient methods \cite{Klintberg2017} and alternating direction method of multipliers (ADMM) \cite{Thuy2017,Dang2015} have received recent attention. ADMM has been shown to be particularly effective in distributed applications (D-ADMM) \cite{Mota2012,Costa2014} where the separability of the cost function can be leveraged. This paper likewise exploits separability of the cost and constraints, although distributed optimization and control are not considered here. We refer the reader to~\cite{Boyd2011} for a full description of ADMM and a survey of applications.

This paper is motivated by the problem of minimizing fuel consumption in a hybrid electric vehicle, achieved by distributing the demanded load between two available power sources (an internal combustion engine and an electric motor) in a blended mode \cite{Sciarretta2007}. The problem is constrained by limits on available power and bounds on the battery state of charge, and MPC is a suitable framework since its feedback mechanism provides a degree of robustness to discrepancies between the predicted and actual power demanded. Although dynamic programming is typically used for the solution of the corresponding optimization \cite{Back2002,Larsson2015}, the computation required for a sufficiently accurate solution prohibits a real-time implementation. A projected Newton method for solving this problem subject to a terminal state of charge constraint was proposed in~\cite{Buerger18}, but the method becomes intractable for the general case of constraints on the state charge at each instant on a prediction horizon. Here, we use ADMM to leverage a separable cost function to enforce constraints at all  future time-steps.

The paper is organised as follows: Section 2 gives the problem definition; Section 3 states the proposed ADMM algorithm; Section 4 describes the energy management problem more fully and discusses its solution for randomly generated examples; Section 5 provides conclusions. Convergence and optimality analyses are included in the appendix.
 
\section{PROBLEM DEFINITION}

Consider a system represented by the discrete time model
\begin{equation*}
x_{k+1} = A_kx_k+b_k(u_k)
\end{equation*}
where $x_k \in \mathbb{R}^{n_x}$ is the system state, $u_k \in \mathbb{R}^{n_u}$ is the control input, $A_k\in \mathbb{R}^{n_x\times n_x}$, and $b_k(\cdot)$ is convex. The states and control inputs are subject to elementwise bounds and the associated optimal control problem is defined 
\begin{equation}\label{eqn_problem_def}
\begin{aligned}
&\underset{x,u}{\text{minimize}} && f(x) + g(u)\\
& \text{subject to} && \begin{rcases} x_{k+1} = A_kx_k+b_k(u_k)  \\
\underline{x}_{k+1} \leq x_{k+1} \leq \overline{x}_{k+1} \\
\underline{u}_k \leq u_k \leq \overline{u}_k
 \end{rcases} 
k = 0,\dots,N-1
\end{aligned}
\end{equation}
where $x=(x_1,\dots,x_N)$ is the vector of future states over a prediction horizon of $N$ steps, $u=(u_0,\dots,u_{N-1})$ is the vector of control inputs over the prediction horizon, and $f,g$ are state and control cost functions. It is assumed that $f$ and $g$ are convex and separable. We say that a function $a(c)$, where $c = (c_1,\dots,c_{N_c})$, is separable if
\begin{equation*}
a(c) = \sum_{i=1}^{N_c} a_i(c_i) .
\end{equation*}
%where $c = (c_1,\dots,c_{N_c})$. 
Note that problem~(\ref{eqn_problem_def}) is in general nonconvex due to the presence of nonlinear dynamics and state constraints.
%The proofs for optimality and convergence provided in \cite{Boyd2011} require the problem to be convex, so we later provide a new proof for this problem class.

\textit{Equality constraints:} $x$ and $x_k$ can be written in terms of
$b(u)=[b_0^\top(u_0)\ \cdots \ b_{N-1}^\top(u_{N-1})]^\top$ as
\begin{equation}\label{eqn_whole_state}
\begin{aligned}
x &= \Phi x_0 + \Psi b(u) \\
x_{k} &= \Phi_k x_0 + \Psi_k b(u) 
\end{aligned}
\end{equation}
 where $\Phi_k$ and $\Psi_k$ are the $k$th block rows of $\Phi$ and $\Psi$:
\begin{equation*}
\Phi = \begin{bmatrix} A_0 \\
\vdots\\
~\prod_{k=0}^{N-1}A_k ~\end{bmatrix} ,
\quad
\Psi = \begin{bmatrix} I & \dots & 0\\
\vdots & \ddots & \vdots\\
~\prod_{k=1}^{N-1}A_k & \dots & I ~\end{bmatrix} .
\end{equation*}
Here $\Psi\in\mathbb{R}^{N n_x \times N n_x}$ is lower-triangular with $ij$th block given by $\prod_{k=j}^{i-1}A_k$ for $j<i$ and by the identity matrix for $i=j$, and the $i$th block-row of $\Phi\in\mathbb{R}^{N n_x \times n_x}$ is $\prod_{k=0}^{i-1}A_k$.

% \textit{Equality constraints:}
% Define $\Psi$ as a $N n_x \times N n_x$ block matrix with diagonal blocks equal to the identity matrix $I$ and off-diagonal lower-triangular blocks given by $\prod_{k=j}^{i-1}A_k$, where $i$ is the block row index and $j$ is the block column index:
% \begin{equation*}
% \Psi = \left[ \begin{array}{ccc} I & \dots & 0\\
% \vdots & \ddots & \vdots\\
% \prod_{k=1}^{N-1}A_k & \dots & I \end{array}
%  \right]
% \end{equation*}
% Define $\Phi$ as an $N n_x \times n_x$ block matrix with each block given by $\prod_{k=0}^{i-1}A_k$, where $i$ is the block row index:
% \begin{equation*}
% \Phi = \left[ \begin{array}{c} A_0 \\
% \vdots\\
% \prod_{k=0}^{N-1}A_k \end{array}
%  \right]
% \end{equation*}
% then, in terms of $b(u)=\left(b_0(u_0),\dots,b_{N-1}(u_{N-1}) \right)^\top$, $x$ can be written as
% \begin{equation}\label{eqn_whole_state}
% x = \Phi x_0 + \Psi b(u)
% \end{equation}
% and $x_{k}$ as
% \begin{equation}\label{eqn_state}
% x_{k} = \Phi_k x_0 + \Psi_k b(u)
% \end{equation}
% where $\Phi_k$ and $\Psi_k$ are respectively the $k$th block rows of $\Phi$ and $\Psi$.

\textit{Inequality constraints:}
Define elementwise indicator functions $h_k^x$,  $h_k^u$ and $h^x$, $h^u$ as
\begin{alignat*}{2}
h_k^x(x_{k}) &= \begin{cases} 0 & x_{k} \in [\underline{x}_{k},\overline{x}_{k}] \\ \infty  & \text{otherwise} \end{cases}
\qquad
h^x(x)&=\sum_{k=1}^{N}h_k^x(x_{k}) , \\
h_k^u(u_k) &= \begin{cases} 0 & u_k \in [\underline{u}_{k},\overline{u}_{k}] \\ \infty  & \text{otherwise} \end{cases}
\qquad
h^u(u) &= \sum_{k=0}^{N-1}h_k^u(u_k) .
\end{alignat*}
Using these definitions we rewrite (\ref{eqn_problem_def}) in a more convenient form without explicit inequality constraints as 
\begin{equation}\label{eqn_problem_def_simplified2}
  \begin{aligned}
    & \underset{x,u}{\text{minimize}} & & f(x) + g(u) + h^x(x) + h^u(u) \\
    & \text{subject to} & & x = \Phi x_0 + \Psi b(u) .
  \end{aligned}
\end{equation}

% \textit{Projection operator:}
% We define the projection $\pi_k^x$ as the minimizer of $d_k(u_k) + h_k^u(u_k)$, where $d_k(u_k)$   
% \begin{equation*}
% \begin{aligned}
% & \arg \min_{u_k} \bigl( d_k(u_k) + h^u_k(u_k) \bigr) \\
% %&\quad\qquad =
% %\arg \underset{u_k}\min (d_k(u_k) \hspace{2mm}  \text{s.t.} \hspace{2mm} u_k \in [\underline{u}_k,\overline{u}_k])\\
% &\quad\qquad =
% \min\bigl\{\overline{u}_k,\max\{\underline{u}_k,\arg\min_{u_k} d_k(u_k) \}\bigr\} \\
% &\quad\qquad =
% \pi^u_k \Bigl[ \arg\min_{u_k}\bigl( d_k(u_k) \bigr) \Bigr] . 
% \end{aligned}
% \end{equation*}
% Similarly, given a convex function $e_k(x_k)$, we define the projection $\pi^x_k$ as
% $\arg \min_{x_k} \bigl( d_k^x(x_k) + h^x_k(x_k) \bigr) = \pi_k^x\bigl[ \arg\min_{x_k}  \bigl( e_k(x_k) \bigr) \bigr]$.
% % \begin{equation*}
% % \begin{aligned}
% % & \arg \underset{x_k}\min \left( d_k^x(x_k) + h^x_k(x_k) \right) \\
% % &\quad\qquad =
% % \arg \underset{x_k}\min (d_k^x(x_k) \hspace{2mm}  \text{s.t.} \hspace{2mm} x_k \in  [\underline{x}_{k}, \overline{x}_{k}])\\
% % &\quad\qquad =
% % \min\{\overline{x}_{k} ,  \max\{\underline{x}_{k}, \arg\underset{x_k}\min(d_k^x(x_k))\}\} \\
% % &\quad\qquad =
% % \pi^x_k \Bigl[ \arg\underset{x_k}\min\left( d_k^x(x_k) \right) \Bigr] . 
% % \end{aligned}
% % \end{equation*}
\section{ADMM ALGORITHM}
Introducing $v$ as a substitute for $b(u)$, we rewrite~(\ref{eqn_problem_def_simplified2}) as
\begin{equation}\label{eqn_minimisation_ADMM}
\begin{aligned}
&\underset{u,x,v}{\text{minimize}} && f(x) + g(u) + h^x(x) + h^u(u)\\
& \text{subject to} &&  \Phi x_0 + \Psi v - x = 0 \\
&&& b(u) - v = 0 .
\end{aligned}
\end{equation}
The associated augmented Lagrangian function is 
\begin{multline}\label{eqn_augmented_lagrangian}
L(u,v,x,y,z) = f(x) + g(u) + h^x(x) + h^u(u) \\
+\frac{\rho_1}{2} \|b(u) - v +y \|^2 + \frac{\rho_2}{2}\|\Phi x_0 + \Psi v - x +z \|^2 , \end{multline}
and the ADMM iteration is obtained (see e.g.~\cite{Boyd2011}) as
\begin{alignat*}{2}
&u_k^{j+1} &\!\!& =  \arg\underset{u_k}\min \Bigl( h_k^u(u_k) + g_k(u_k) + \frac{\rho_1}{2} (b_k(u_k) \!-\! v^j_k \!+\! y^j_k )^2 \Bigr) \\
&v^{j+1} &\!\!& = \begin{aligned}[t] \arg \underset{v} \min \Bigl( \frac{\rho_1}{2} \|b(u^{j+1}) &- v +y^j \|^2 \\ 
&\!\! + \frac{\rho_2}{2}\|\Phi x_0 + \Psi v - x^j +z^j \|^2 \Bigr) \end{aligned} \\
&x_{k+1}^{j+1} &\!\!& = \begin{aligned}[t] \arg \underset{x_{k+1}} \min \Bigl(
& h_k^x(x_{k+1}) + f_{k+1}(x_{k+1}) \\
& \!\! +\frac{\rho_2}{2}(\Phi_{k+1} x_0 + \Psi_{k+1} v^{j+1}\!- x_{k+1} \!+ \! z^j_k )^2 \Bigr) \end{aligned} \\
&y^{j+1} &\!\!& = y^j + b(u^{j+1}) - v^{j+1} \\
&z^{j+1} &\!\!& = z^j + \Phi x_0 + \Psi v^{j+1}-x^{j+1}
\end{alignat*}
for $k = 0,\dots,N-1$, where $j$ is the iteration counter. We assume that a suitable solver is available for updating $u^{j+1}$ and $x^{j+1}$. For the application considered in Section~\ref{section_simulations} analytical solutions exist for the minimizers $u_k^{j+1}$ and $x_{k+1}^{j+1}$, which are computed by finding the roots of cubic equations and projecting these onto the inequality constraints in~(\ref{eqn_problem_def}).
%; hence analytical solutions exist in this case.
%, and hence the global minima can be found efficiently.

The update for $v$ can be written explicitly as
\[
  \begin{aligned}
    v^{j+1} =  \left( \rho_1I+\rho_2 \Psi^\top \Psi \right)^{-1}\bigl[ & \rho_1(b(u^{j+1}) + y^j) \\
    &+ \rho_2\Psi^\top(-\Phi x_0 +x^j-z^j)\bigr] .
  \end{aligned}
\]
For the application considered in Section \ref{section_simulations}, where ${A_k = 1}$, the matrix $\Psi$ is a lower triangular matrix of 1s, so $(\rho_1I+\rho_2\Psi^\top\Psi)^{-1}$ becomes a linear time-invariant filter. Approximating this matrix by setting the elements that do not exceed a given threshold to zero then yields a banded matrix, making the computation of $v^{j+1}$ straightforward. For the general case in which the matrices $A_k$ defining $\Psi$ are arbitrary, the results in \cite{Klintberg2017} and \cite{Demko1984} provide conditions under which $(\rho_1 I + \rho_2 \Psi^\top \Psi)^{-1}$ is approximately banded and bound the rates of decay of elements with distance from the diagonal. 
%These bounds are likely to be conservative whenever the system $x_{k+1} = A_k x_k$ is not strictly stable, and a useful area of further investigation would be to determine improved bandwidth estimates.

The iteration is terminated when primal and dual residual variables have fallen below pre-defined thresholds,
$\epsilon^{\mathrm{primal}}$ and $\epsilon^{\mathrm{dual}}$,
chosen based on the required accuracy of the optimization and the typical magnitudes of decision variables:
\[
\|r^{j+1} \|_2 \leq \epsilon^{\mathrm{primal}} , \quad
\|s^{j+1} \|_2 \leq \epsilon^{\mathrm{dual}} .
\]
% \begin{equation*}
% \begin{aligned}
% \|r^{j+1} \|_2 &\leq \epsilon^{\mathrm{primal}} \\
% \|s^{j+1} \|_2 &\leq \epsilon^{\mathrm{dual}}
% \end{aligned}
% \end{equation*}
In the appendix we analyse the optimality and convergence properties of the algorithm, and provide definitions of the residuals $r^{j}$ and $s^j$. 
%We also give conditions for the iteration to converge to a point  
We also give conditions for the algorithm to converge to the global minimum, despite the problem being nonconvex, however this condition cannot generally be determined \textit{a priori}, and the algorithm may converge to a local minimum, or even a maximum, if it is not met.

\section{APPLICATION TO PHEV SUPERVISORY CONTROL}\label{section_simulations}
We are interested in the problem of minimizing the fuel consumption of a plug-in hybrid electric vehicle (PHEV) by controlling the power balance between the electric motor and internal combustion engine in a blended mode. The power, $d$, demanded by the driver is split between the motor $m$ and engine $u$ so that $u+m = d$. The system can be modelled as 
\begin{equation*}
x_{k+1} = x_k+b_k(u_k)
\end{equation*}
where $x_k$ is the state of charge of the battery, $b_k$ accounts for the charge  loss-dynamics, and the engine power $u_k$ is the control input (i.e.~$n_x=n_u = 1$ and $A_k = 1$). The fuel minimization problem over a horizon of $N$ time steps is 
\begin{equation}\label{eq:phev_opt}
\begin{aligned}
&\underset{u}{\text{minimize}} &&  \sum_{k=0}^{N-1} g_k(u_k) \\
& \text{subject to}  && \begin{rcases} x_{k+1} = x_k+b_k(u_k) \\
\underline{x}_{k+1} \leq x_{k+1} \leq \overline{x}_{k+1} \\
\underline{u}_k \leq u_k \leq \overline{u}_k
\end{rcases} k = 0,\dots,N-1
\end{aligned}
\end{equation}
where $g_k$ is the fuel consumption determined from a quasi-static model at time-step $k$; $\overline{u}_{k}$ and $\underline{u}_{k}$ represent upper and lower limits on the engine power  output; and $\underline{x}_{k+1}$ and $\overline{x}_{k+1}$ represent the bounds on state of charge. The cost and input loss functions are represented by time-varying quadratic functions as
\begin{equation*}
\begin{aligned}
g_k(u_k) &= \alpha_{2,k}u_k^2 + \alpha_{1,k}u_k + \alpha_{0,k} \\
b_k(u_k) &= -\beta_{2,k}(d_k - u_k)^2 - \beta_{1,k}(d_k - u_k) - \beta_{0,k} 
\end{aligned}
\end{equation*}
See \cite{Buerger18} for details on how these functions are obtained from the fuel map and electrical loss maps for a particular PHEV.

To demonstrate the performance of the ADMM algorithm without reference to a particular PHEV powertrain, we generate random sets of nominal systems of the same form as the motivating problem in~(\ref{eq:phev_opt}). We generate uniformly distributed random disturbance values $d_k \in [-1,1]$; uniformly distributed $\alpha$ and $\beta$ coefficients $\alpha_{2,k},\beta_{2,k} \in [0,0.1]$, $\alpha_{1,k},\beta_{1,k} \in [0,1]$, and $\alpha_{0,k},\beta_{0,k} =0$, and a uniformly distributed initial state $x_0 \in [-0.5,0.5]$. We use $\underline{u}_{k} =-0.5$ and $\overline{u}_{k} =0.5$ as control input limits, and state limits of $\underline{x}_k = -2$ and $\overline{x}_k = 2$. The ADMM iteration is initialised assuming that the inequality constraints in~(\ref{eq:phev_opt}) are inactive:
%initial values for the ADMM algorithm are calculated assuming that the inequality constraints in state constraints are inactive as:
%\begin{gather*}
\begin{align*}
u_k^0 &=  \pi_k^u\Bigl[ \arg\underset{u_k}\min \hspace{2mm}  g_k(u_k) \Bigr], 
\quad
& v^0 & = b(u^0),
\\
x_{k+1}^0 &= \pi _{k+1} ^x \Bigl[  \Phi_{k+1} x_0 + \Psi_{k+1} v^0   \Bigr],
\quad
& y^0 & = 0,
\quad 
z^0 = 0,
\end{align*}
% \begin{equation*}
% \begin{aligned}
% u_k^0 &=  \pi_k^u\Bigl[ \arg\underset{u_k}\min \hspace{2mm}  g_k(u_k) \Bigr] \\
% v^0 & = b(u^0) \\
% x_{k+1}^0 & = \pi _{k+1} ^x \Bigl[  \Phi_{k+1} x_0 + \Psi_{k+1} v^0   \Bigr]  \\
% y^0 & = 0\\
% z^0 & = 0
% \end{aligned}
% \end{equation*}
for $k = 0,\dots,N-1$, 
where $\pi_k^u$, $\pi_k^x$ are projection operators:
\begin{align*}
\pi_k^u(u_k) &= \min\bigl\{\overline{u}_k,\max\{\underline{u}_k,u_k \}\bigr\}  \\
%\quad \text{and} \quad 
\pi_k^x(x_k) &=\min\bigl\{\overline{x}_{k} ,  \max\{\underline{x}_{k}, x_k\}\bigr\} .
\end{align*}
If $\Phi x_0 + \Psi v^0 = x^0$, then these initial values are optimal and the state constraints are inactive. The stopping conditions used are equal i.e $\epsilon^{\mathrm{primal}}=\epsilon^{\mathrm{dual}}=\epsilon$. The simulations are run in Matlab, using an Intel 2.60 GHz i7-6700HQ CPU.

\subsection{Variation of iteration count with $\rho$}
The rate of convergence of the ADMM algorithm is affected by the values of $\rho_1$ and $\rho_2$. For the algorithm to be practically useful, it is desirable for the optimal values for these parameters to be independent of both horizon length and system parameters, so that a single set of values can be used in all conditions. Here we demonstrate the effect of $\rho_1$ and $\rho_2$ on the convergence properties of the algorithm for 20 randomly generated systems. Figure~\ref{fig_3a} shows the average numbers of iterations required for completion of the algorithm for horizon lengths of 50, 100, 200, and 400 with $\epsilon = 0.001$, $10^{-1}\leq \rho_1\leq 10^2$ and $10^{-2}\leq \rho_2\leq 10$.

%The results of The ADMM algorithm with $\epsilon = 0.001$, $10^{-1}\leq \rho_1\leq 10^2$ and $10^{-2}\leq \rho_2\leq 10$ to determine the average numbers of iterations required for completion of the algorithm for horizon lengths of 50, 100, 200, and 400 (Fig.~\ref{fig_3a}). 
%To enhance the clarity of the results, regions where the average number of iterations is greater than 300 are blacked out.

For each horizon length, the minimum number of iterations for convergence is obtained with $1 \leq \rho_1 \leq 10$ and $0.1 \leq \rho_2 \leq 1$. Moreover, variations in $\rho$ within an order of magnitude of this minimum do not produce large increases in total number of iterations, implying a degree of robustness. Therefore $\rho$ values within these ranges provide close to the minimum number of iterations for horizon lengths between 50 and 400, suggesting that a single set of values is appropriate for the PHEV problem.
Although the optimal $\rho$ values are similar, the minimum number of iterations for completion increases significantly with horizon length. This is explored further in the following sections.

\begin{figure}[h!]%[H]
%\vspace{-4mm}
\centerline{\includegraphics[scale=0.72]{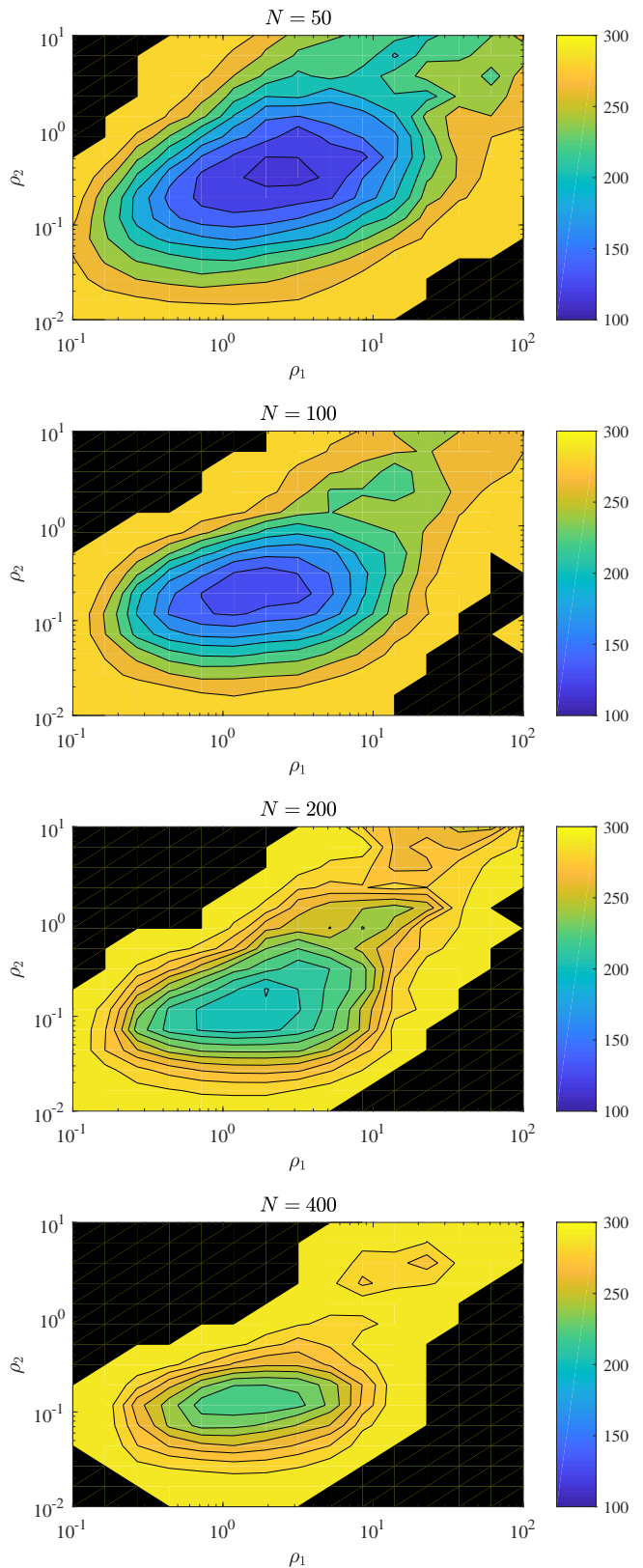}}
\vspace{-3mm}
\caption{Variation with $\rho_1,\rho_2$ of the number of iterations before termination. The data presented is the arithmetic mean for 20 generated systems. }
\label{fig_3a}
\vspace{-5mm}
\end{figure}

\subsection{Variation of required number of iterations with $\epsilon$}
The threshold $\epsilon$ determines the degree of optimality at termination, since $(u,x) \rightarrow (u^*,x^*)$ as $\epsilon \rightarrow 0$ (see Appendix A). However, a trade-off exists between accuracy and computation since small values of $\epsilon$ require many iterations. Here we investigate how the number of iterations before termination varies with $\epsilon$ for 200 randomly generated problems with $N=100$. Figure~\ref{fig_1} shows the number of iterations required with $10^{-4}\leq\epsilon\leq 10^0$, $\rho_1=1$, and $\rho_2=0.2$.

Clearly, for a low-accuracy solution only a few tens of iterations are required, but this rises rapidly for high-accuracy solutions. For example, with $\epsilon =0.14$  the median iteration count is 29, but $\epsilon = 10^{-4}$ requires a median iteration count of 210. Furthermore, the increase in iteration count is highly nonlinear as $\epsilon$ is reduced, and for small $\epsilon$ modest further reductions have a large effect on the iteration count. 

The variation in the number of iterations required increases as $\epsilon$ is reduced (Figure~\ref{fig_1}): the 98th percentile is 45 for $\epsilon =0.14$ ($\sim$ 2 times the median), rising to 2374 ($\sim$ 10 times the median) for $\epsilon = 10^{-4}$. Again, this increase in uncertainty increases nonlinearly as $\epsilon$ is reduced. Furthermore, the distribution of the data is heavily skewed, and outliers become more extreme as $\epsilon$ is reduced. This is problematic as the MPC optimization algorithm must be designed for the worst case, and at low values of $\epsilon$ such a high potential iteration count could be extremely prohibitive.

These properties suggest that the ADMM algorithm may be appropriate for finding an approximate solution to the PHEV problem that is then used to initialise another method, but inappropriate for obtaining an exact solution. For example, the ADMM algorithm could be used to obtain the active set, then an alternative method (such as that proposed in~\cite{Buerger18}) used to obtain the corresponding optimum. This observation is application-specific as large numbers of iterations may not be problematic if the system dynamics are sufficiently slow or computational resources are sufficiently high.

\begin{figure}%[H]
\vspace{2mm}
\centerline{\includegraphics[scale=0.675]{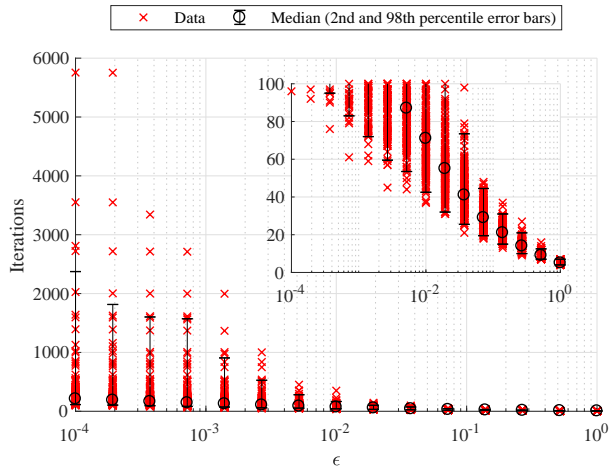}}
\vspace{-3mm}
\caption{Number of iterations required for completion as $\epsilon$ varies.}
\label{fig_1}
\vspace{-5mm}
\end{figure}

\subsection{Variation of computation with horizon length}
The MPC framework can be implemented in the PHEV energy management problem with either a receding horizon (with fixed $N$) or a shrinking horizon (with $N$ reducing as the vehicle progresses through a journey). Although the algorithm can be tuned for a specific horizon length in the receding horizon case, for a shrinking horizon implementation it must be robust to a range of horizon lengths. Thus, for a 1 hour journey sampled at 1\,s, $N$ varies between 1 and 3600. Here we demonstrate the effect of horizon length on number of iterations and time to termination for 200 systems, with $N$ varying between 10 and 400. We choose $\epsilon=10^{-2}$, and $\rho_1=1$, $\rho_2=0.2$. The results are shown in Figure \ref{fig_3}.

For small $N$, increases in horizon length produce large increases in the number of iterations (the median for $N=25$ is 22, increasing to 40 when $N = 75$), but this effect tapers off quickly, and the median iteration count is 105 for $N = 350$. There is no clear dependence on $N$ of the uncertainty in the required number of iterations; the 98th percentile has a maximum value of 273 iterations at $N=80$, a minimum of 85 at $N=10$, and large variations between these values.

Although the iteration count does not increase significantly with longer horizons, the total time taken to find the solution increases roughly linearly with horizon length, from a median of 0.0075\,s for $N=10$, to a median of 1.39\,s for $N=400$. The uncertainty also appears to increase within a near-linear envelope. This is due to the computation required to update $u^{j+1}$, for which $N$ cubic equations must be solved at each iteration. This suggests that for long prediction horizons, improving the solution times for $u^{j+1}$ and $x^{j+1}$ would have a significant effect on total solution time.
% than improving the ADMM algorithm.
%
Note that the absolute computation times presented here demonstrate the relationship between horizon length and computation, but 
%should not be used to assess viability for the PHEV problem, only the relationship between horizon length and computation time. The 
the hardware used for the simulations does not correlate to that typically found in a PHEV, and Matlab is not typically used for embedded computing.

\begin{figure}%[H]
\centering
\vspace{2mm}\includegraphics[scale=0.68]{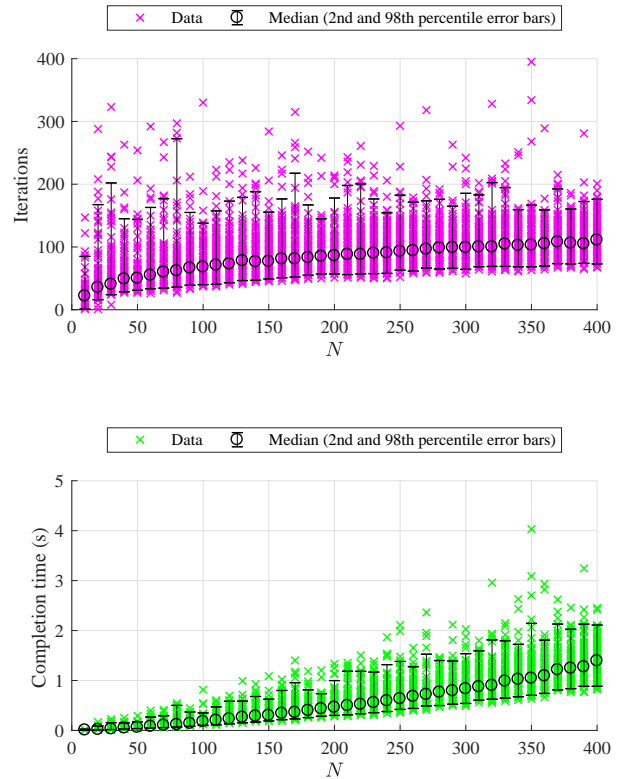}
\vspace{-3mm}
\caption{Iterations and time required for completion against $N$.}
\label{fig_3}
\vspace{-5mm}
\end{figure}

\section{CONCLUSIONS}

%In this paper we consider a general problem of 
A general MPC optimization problem is considered for systems with separable cost functions, nonlinear input map and simple constraints, motivated by the blended mode energy management problem for PHEVs. We propose an ADMM algorithm for solving this problem and demonstrate conditions for convergence to a solution satisfying first order necessary conditions for optimality.
%and optimality, but cannot guarantee convergence to the global minimum for the demonstrated problem class. 
Numerical experiments show that an approximate solution is reached within a few tens of iterations, while significantly more are required for an exact solution. The effect of algorithm parameters on the number of iterations and solution times is also investigated.
Future work will investigate methods for implementing the proposed algorithm in a PHEV, to obtain globally optimal control inputs that are implementable in real time.  
%%%%%%%%%%%%%%%%%%%%%%%%%%%%%%%%%%%%%%%%%%%%%%%%%%%%%%%%%%%%%%%%%%%%%%%%%%%%%%%%

\section*{Appendix}

\subsection{Optimality}
We use a similar method to that given in Section 3 of \cite{Boyd2011}, however we make use of an alternative Lyapunov function to demonstrate convergence, as required by the nonconvexity of the problem. Problem (\ref{eqn_minimisation_ADMM}) is first rewritten as 
\begin{equation}\label{eqn_proof_opt}
\begin{aligned}
& \underset{u}{\text{minimize}} \hspace{2mm} && \hat{f}(\hat{x}) + \hat{g}(u)  \\
& \text{subject to} && \hat{b}(u) + B\hat{x} = 0
\end{aligned}
\end{equation}
where
\begin{equation*}
\begin{aligned}
&\hat{f}(\hat{x}) = f(x) + h^x(x), \ \hat{g}(u) = g(u) + h^u(u)\\
&\hat{b}(u) = \left[ \begin{array}{c} b(u) \\ \Phi x_0 \end{array} \right], \ 
B = \left[ \begin{array}{cc} -I & 0 \\ \Psi & -I \end{array} \right], \
\hat{x} = \left[ \begin{array}{c} v \\ x \end{array} \right] .
\end{aligned}
\end{equation*}
Then, by expressing the augmented Lagrangian (\ref{eqn_augmented_lagrangian}) as
\begin{equation*}
  L(u,\hat{x},\hat{y}) = \hat{f}(\hat{x}) + \hat{g}(u) + \hat{y}^\top(\hat{b}(u)+B\hat{x})  + \tfrac{1}{2}\| \hat{b}(u)+B\hat{x} \|^2_R
\end{equation*}
where $\| w \|^2_R = w^\top R w$ for any vector $w$ of conformal dimensions, and defining $\hat{y} = [y^\top \ \ z^\top]^\top$, $R = \diag\{\rho_1I,\rho_2 I\}$,
% \[
%   \hat{y} = \begin{bmatrix} y \\ z \end{bmatrix},
%   \qquad
% R = \begin{bmatrix} \rho_1 I & 0 \\ 0 & \rho_2I \end{bmatrix} ,
% \]
the ADMM iteration becomes
\begin{subequations}\label{eqn_proof_iteration}
\begin{align}
  u^{j+1} &= \underset{u}{\arg \min} \Bigl\{ \hat{g}(u) + (\hat{y}^j)^\top(\hat{b}(u) + B\hat{x}^j)  \nonumber 
  \\
          &\hspace{30.5mm} + \tfrac{1}{2}\| \hat{b}(u)+B\hat{x}^j \|^2_R \Bigr\}
  \\
  \hat{x}^{j+1} &= \underset{\hat{x}}{\arg \min} \Bigl\{ \hat{f}(\hat{x}) + (\hat{y}^j)^\top(\hat{b}(u^{j+1}) + B\hat{x})  \nonumber
  \\
          &\hspace{27mm} + \tfrac{1}{2}\| \hat{b}(u^{j+1})+B\hat{x} \|^2_R \Bigr\}    \\
  \hat{y}^{j+1} & =  
                  \hat{y}^j + R(\hat{b}(u^{j+1})+B\hat{x}^{j+1})
\end{align}
\end{subequations}
\begin{assumption}
Slater's condition holds, i.e.\ $u,v,x$ exist such that
$\underline{u}_k  < u_k < \overline{u}_k$, $\underline{x}_{k+1} < x_{k+1} < \overline{x}_{k+1}$ for all $k$, where $x = \Phi x_0 + \Psi v$ and $v = b(u)$.

\end{assumption}

The first order conditions satisfied by a (possibly local, possibly maximal) optimal solution $(u^*,\hat{x}^*,\hat{y}^*)$ are
\begin{subequations} \label{subequations_opt}
\begin{align}
& 0 =\hat{b}(u^*) + B\hat{x}^* \label{subeqn_opt_1} \\ 
& 0 = \nabla \hat{g}(u^*) + [\partial\hat{b}(u^*)]^\top\hat{y}^* \label{subeqn_opt_2} \\ 
& 0 = \nabla \hat{f}(\hat{x}^*) + B^\top\hat{y}^* \label{subeqn_opt_3}
\end{align}
\end{subequations}
where $\nabla \hat{g}$, $\nabla \hat{f}$ are the gradients of $\hat{g}$, $\hat{f}$, and $\partial \hat{b}$ is the Jacobian matrix of $\hat{b}$.

\begin{lemma}\label{proposition_1}
\textit{$(u^{j+1},\hat{x}^{j+1},\hat{y}^{j+1})$ satisfies the first order necessary  optimality conditions~(\ref{subequations_opt}) iff $r^{j+1}=0$ and $s^{j+1}=0$.}
\end{lemma}

\begin{proof}
First consider condition (\ref{subeqn_opt_3}). Since $\hat{x} = \hat{x}^{j+1}$ is by definition the minimizer of $L(u^{j+1},\hat{x},\hat{y}^j)$, we have
\begin{equation*}
\nabla \hat{f}(\hat{x}^{j+1}) + B^\top \hyj + B^\top R (\hat{b}(u^{j+1}) + B \hxjp ) = 0
\end{equation*}
and hence the update law $\hyjp = \hyj + R(\hat{b}(u^{j+1}) + B\hxjp)$ implies $\nabla \hat{f}(\hxjp) + B^\top \hyjp = 0$ for all $j$. Next, consider (\ref{subeqn_opt_2}). By definition $u = u^{j+1}$ minimizes $L(u,\hxj,\hyj)$ and therefore $u^{j+1}$ necessarily satisfies
\[%\begin{align*}
\nabla \hat{g}(u^{j+1}) + [\partial \hat{b}(u^{j+1}) ]^\top \bigl(\hyj  + R (\hat{b}(u^{j+1}) + B\hxj)\bigr) = 0
\]%\end{align*}
Hence the update law for $\hyjp$ gives
\[%\begin{align*}
\nabla \hat{g}(u^{j+1}) + [\partial \hat{b}(u^{j+1}) ]^\top \bigl( \hyjp + RB (\hxj - \hxjp) = 0
\]%\end{align*}
and we have, for all $j$,
\[%\begin{align*}
\nabla \hat{g}(u^{j+1}) + [\partial \hat{b} (u^{j+1})]^\top \hyjp + s^{j+1} = 0
\]%\end{align*}
Therefore, the ADMM iteration converges to a point satisfying the first order necessary conditions (\ref{subeqn_opt_1})-(\ref{subeqn_opt_3}) if and only if the residual variables defined by
\begin{align*}
r^{j+1} &= \hat{b}(u^{j+1}) + B \hxjp \\
s^{j+1} &= [\partial \hat{b}(u^{j+1}) ]^\top RB (\hxj - \hxjp)
\end{align*}
converge to zero.
\end{proof}

\subsection{Convergence}\label{section_convergence}

Let $(u^\O,\hat{x}^\O,\hat{y}^\O)$ denote a solution of (\ref{eqn_proof_opt}) that achieves the globally minimum value of the objective. Let $L_0(u,\hx,\hy) = \hat{f}(\hx) + \hat{g}(u) + \hy^\top (\hat{b}(u) + B \hx )$ and define $\S_c$ as the set
\[%\begin{align*}
\S_c \!=\! \bigl\{\hspace{-0.5pt} (u,\hx,\hy) : \| \hy - \hy^\O \hspace{-1pt}\|^2_{R^{-1}}  \! + \|B(\hx - \hx^\O\hspace{-1pt}) \|^2_R + \|r(u,\!\hx)\|^2_R \hspace{-0.5pt}\leq \hspace{-0.5pt} c \hspace{-0.5pt}\bigr\}
\]%\end{align*}
for $c>0$, where $r(u,\hx) = \hat{b}(u) + B \hx$.
\begin{assumption}
$\S_c$, $(u^\O,\hx^\O,\hy^\O)$, $(u^j,\hx^j,\hy^j)$ are such that:
\begin{enumerate}[label=\alph*)]
\item iteration (\ref{eqn_proof_iteration}) is initialised at $(u^0,\hx^0,\hy^0) \in \S_c$, i.e.\\ 
$ c \geq \| \hy^0 - \hy^\O \|^2_{R^{-1}} + \|B(\hx^0 - \hx^\O \|^2_R + \|r(u^0,\hx^0)\|^2_R$,
\item $u = u^{j+1}$ in (\ref{eqn_proof_iteration}a) is the global minimiser of $L(u,\hx^j,\hy^j)$.
% \item $\hx = \hx^{j+1}$ defined in (\ref{eqn_proof_iteration}b) is the global minimiser of $L(u^{j+1},\hx,\hy^j)$ (necessary due to the convexity of the $\hx$ update step).
\end{enumerate}
\end{assumption}
\begin{theorem} \label{proposition_2}
\textit{$(u^j,\hx^j,\hy^j) \in \S_c$,
 for all $j \in \mathbb{Z}^+$, and $r^j \rightarrow 0$ and $s^j \rightarrow 0$ as $j\rightarrow \infty$.}
\end{theorem}
\begin{proof}
We first show that $V^j = V(u^j,\hx^j,\hy^j)$ is a Lyapunov function, where $V$ is defined by
\begin{align*}
V(u,\hx,\hy) = \| \hy - \hy^\O \|^2_{R^{-1}}  + \|B(\hx - \hx^\O )\|^2_R + \|r(u,\hx)\|^2_R .
\end{align*}
This part of the proof is split into three steps:
\begin{enumerate}[label = \roman*)]
\item The update law for $u^{j+1}$ can be equivalently written as
\begin{equation*}
u^{j+1} = \underset{u}{\arg \min}\!\left\{ \hat{g}(u) \!+\! ( \hyjp \hspace{-3pt}+\! RB(\hxj \!-\! \hxjp))^{\!\top} \hat{b}(u) \right\}
\end{equation*}
and, from Assumption 3(b) it follows that 
\begin{align*}
\hat{g}(u^{j+1}) &+ (\hat{y}^{j+1} + RB(\hxj - \hxjp))^\top \hat{b}(u^{j+1}) \\ 
&\leq \hat{g}(u^{\O}) + (\hat{y}^{j+1} + RB(\hxj - \hxjp))^\top \hat{b}(u^{\O}) .
\end{align*}
Similarly, $\nabla \hat{f}(\hxjp) + B^\top \hyjp = 0 $ implies that
\begin{equation*}
\hat{f}(\hxjp ) + (\hyjp )^\top B \hxjp \leq \hat{f}(\hat{x}^\O ) + (\hyjp )^\top B \hat{x}^\O .
\end{equation*}
Therefore, defining $p^{j+1} = \hat{f} ( \hxjp ) + \hat{g}(u^{j+1})$ and $p^\O = \hat{f}(\hat{x}^\O) + \hat{g}(u^\O)$, we obtain
\begin{align*}
& p^{j+1} - p^\O = \hat{f}(\hxjp) - \hat{f}(\hat{x}^\O) + \hat{g}(u^{j+1}) - \hat{g}(u^\O) \\
&\leq (\hyjp)^\top (B(\hat{x}^\O - \hxjp) + \hat{b}(u^\O) - \hat{b}(u^{j+1})) \\
&\quad + (\hxj - \hxjp )^\top B^\top R (\hat{b}(u^\O ) - \hat{b}(u^{j+1})) \\
&= -(\hyjp)^\top r^{j+1} \\ 
&\quad + (\hxj - \hxjp)^\top B^\top R (-r^{j+1} + B(\hxjp - \hat{x}^\O))
\end{align*}

\item Since $(u^\O,\hat{x}^\O,\hat{y}^\O)$ must satisfy the first order necessary conditions (\ref{subequations_opt}), it follows that $u=u^\O$ is the minimiser of $\hat{g}(u) + (\hat{y}^\O)^\top \hat{b}(u)$ and $\hx = \hx^\O$ is the minimiser of $\hat{f}(\hx ) + ( \hy^\O)^\top B \hx$. Therefore, we have, for all $u$ and $\hx$,
\begin{align*}
\hat{g}(u^\O ) + (\hy^\O)^\top \hat{b}(u^\O ) &\leq \hat{g}(u) + (\hy^\O )^\top \hat{b} (u) \\
\hat{f}(\hx^\O ) + (\hy^\O )^\top B \hx^\O &\leq \hat{f}(\hat{x}) + (\hy^\O )^\top B \hx ,
\end{align*} 
so that $\hat{f}(\hx^\O ) + \hat{g} ( u^\O) \leq \hat{f} (\hx ) + \hat{g} (u) + (\hy^\O )^\top (B\hx + \hat{b}(u))$ and hence $p^\O - p^{j+1} \leq (\hy^\O )^\top r^{j+1}$.

\item Combining the bounds on $p^\O - p^{j+1}$ in i) and ii) yields
\begin{multline}\label{eqn_proof_ineq}
%\begin{aligned}
(\hy^{j+1} - \hy^\O )^\top r^{j+1} - (\hxjp - \hxj )^\top B^\top R r^{j+1} \\
+ (\hxjp - \hxj )^\top B^\top R B (\hxjp - \hx^\O) \leq 0 .
%\end{aligned}
\end{multline}
The first term in (\ref{eqn_proof_ineq}) can be simplified using the update law $\hyjp = \hyj + R r^{j+1}\!$ and completing the square:
\begin{align*}
& 2(\hyjp - \hy^\O )^\top r^{j+1} = 2(\hyj -\hy^\O)^\top r^{j+1} + 2 \| r^{j+1} \|^2_R \\
&= \| \hyj - \hy^\O\! + R r^{j+1} \| ^2 _{R^{-1}}\! - \| \hyj - \hy^\O \|^2_{R^{-1}} + \!\| r^{j+1} \|^2_R \\
&= \| \hyj - \hy^\O\! + \hy^{j+1}\! - \hy^j \| ^2 _{R^{-1}}\hspace{-3pt} - \!\| \hyj - \hy^\O \|^2_{R^{-1}}\hspace{-3pt} +\! \| r^{j+1} \|^2_R \\
&= \|\hy^{j+1} - \hy^\O \|^2_{R^{-1}}\! - \| \hy^j - \hy^\O \|^2_{R^{-1}}\! + \|r^{j+1} \|^2_R .
\end{align*}
To simplify the 2nd and 3rd terms in (\ref{eqn_proof_ineq}) we use
\begin{align*}
&\|r^{j+1} \|^2_R - 2(\hxjp - \hxj)^\top B^\top R r^{j+1} \\ 
&\qquad + 2(\hxjp - \hxj )^\top B^\top R B (\hxjp - \hx^\O) \\
&= \|r^{j+1} - B(\hxjp - \hxj) \|^2_R  \\
&\qquad + \| B(\hxjp - \hx^\O ) \|^2_R -\| B(\hxj - \hx^\O ) \|^2_R .
\end{align*}
Therefore (\ref{eqn_proof_ineq}) is equivalent to
\begin{multline*}
\| \hy^{j+1} - \hy^\O \|^2_{R^{-1}} - \| \hy^j - \hy^\O \|^2_{R^{-1}} \\
+ \|r^{j+1} - B(\hx^{j+1} - \hx^j) \|^2_R + \|B(\hxjp - \hx^\O ) \|^2_R \\
 - \| B(\hxj - \hx^\O ) \|^2_R \leq 0
\end{multline*}
which, using the definition of $V^j = V(u^j,\hx^j,\hyj)$, is equivalent to 
\begin{align*}
V^{j+1} - V^j &\leq - \|r^{j+1} - B(\hxjp - \hxj) \|^2_R \\
&\quad - \|r^j \|^2_R + \|r^{j+1}\|^2_R \\
& =  -\|r^j \|^2_R - \|B(\hxjp - \hxj) \|^2_R \\
&\quad +2(r^{j+1})^\top R B (\hxjp - \hxj) \\
& = -\| r^j \|^2_R - \| B(\hxjp - \hxj) \|^2_R \\
&\quad + 2(\hyjp - \hyj)^\top B (\hxjp - \hxj) .
\end{align*}
To bound the RHS of this expression, we recall from i) that $\hx = \hxj$ and $\hx = \hxjp$ are the minimizers of $\hat{f}(\hx) + (\hyj)^\top B\hx$ and $\hat{f}(\hx) + (\hyjp)^\top B \hx$, so that
\begin{align*}
\hat{f}(\hxj) + (\hyj)^\top B \hxj \leq \hat{f} ( \hxjp ) + (\hyj)^\top B \hxjp
\end{align*}
and
\begin{align*}
\hat{f}(\hxjp) + (\hyjp)^\top B \hxjp \leq \hat{f} ( \hxj ) + (\hyjp)^\top B \hxj
\end{align*}
By summing these two inequalities we obtain $(\hyjp - \hyj)^\top B (\hxjp - \hxj) \leq 0$, and it follows that 
\begin{equation*}
V^{j+1} - V^j \leq - \|r^j \|^2_R - \| B(\hxjp - \hxj ) \|^2_R .
\end{equation*}
\end{enumerate}
The positive invariance of $\S_c$ follows from its definition as a level set of $V(u,\hx,\hy)$ and from the fact that $V^j$ is monotonically non-increasing with $j$. Asymptotic convergence of the residuals $r^j$ and $s^j$ can be established by summing both sides of the inequality satisfied by $V^{j+1} - V^j$ over all $j \geq 0$:
\begin{equation*}
\sum_{j=0}^\infty \left( \|r^j\|^2_R + \|B(\hxjp - \hxj) \|^2_R \right) \leq V^0.
\end{equation*}
From this bound we can conclude that, as $j \rightarrow \infty$, $r^j \to 0$ and ${B(\hxjp - \hxj) \rightarrow 0}$, and hence also $s^j \rightarrow 0$.
\end{proof}

Lemma \ref{proposition_1} and Theorem \ref{proposition_2} imply that the ADMM iteration converges to a point satisfying the first order necessary optimality conditions, and this will be the global minimum if the problem is convex. The invariant nature of $\S_c$ also implies that even when the problem is nonconvex, the algorithm will converge to the global minimum if it is initialised within a sub-level set of $V$ that contains no local minima or maxima. This condition is, however, impossible to demonstrate \textit{a priori}, and if the algorithm is initialised outside of this set it could converge to a local minimum or even a local maximum point. We note finally that the Theorem~\ref{proposition_2} also demonstrates that the iteration necessarily terminates after a finite number of steps whenever the threshold $\epsilon$ on $r^{j+1}$ and $s^{j+1}$ is set to a fixed non-zero value.

%%%%%%%%%%%%%%%%%%%%%%%%%%%%%%%%%%%%%%%%%%%%%%%%%%%%%%%%%%%%%%%%%%%%%%%%%%%%%%%%

%\bibliographystyle{IEEEtran}
\bibliographystyle{unsrt}
\bibliography{new_library}

\end{document}